\documentclass[12pt]{article}
\usepackage{amsfonts}
\usepackage{amsmath}

\newtheorem{theorem}{Theorem}

\newtheorem{lemma}[theorem]{Lemma}

\newtheorem{proposition}[theorem]{Proposition}
\newtheorem{remark}[theorem]{Remark}

\newenvironment{proof}[1][Proof]{\noindent\textbf{#1.} }{\ \rule{0.5em}{0.5em}}
\newdimen\dummy
\dummy=\oddsidemargin
\begin{document}

\title{Limits of quotients of real analytic functions in two variables}
\author{ Carlos A. Cadavid*, Sergio Molina**, Juan D. V\'{e}lez**, \\
\small {* Corresponding Author, Universidad EAFIT,}\\
\small{Departamento de Ciencias B\'{a}sicas,}\\
\small{Bloque 38, Office 417. Carrera 49 No. 7 Sur -50,}\\
\small{Medell\'in, Colombia, ccadavid@eafit.edu.co, (57) (4)-2619500. Fax (57) (4)-3120649.} \\
\small {** Universidad Nacional de Colombia, Departamento de Matem\'{a}ticas,}\\
\small {Calle 59A No 63 - 20 , Oficina 43-106, Medell\'{\i}n, Colombia,}\\
\small{sdmolina@gmail.com, jdvelez@unal.edu.co} }

\maketitle

\begin{abstract}
Necessary and sufficient conditions for the existence of limits of the form 
\begin{equation*}
\lim_{(x,y)\rightarrow (a,b)}\frac{f(x,y)}{g(x,y)}
\end{equation*}
are given, under the hipothesis that $f$ and $g$ are real analytic functions
near the point $(a,b)$, and $g$ has an isolated zero at $(a,b)$. An
algorithm (implemented in MAPLE 12) is also provided. This algorithm
determines the existence of the limit, and computes it in case it exists. It
is shown to be more powerful than the one found in the latest versions of
MAPLE. The main tools used throughout are Hensel's Lemma and the theory of
Puiseux series.
\end{abstract}


\section{Introduction}

In the usual calculus courses one is asked to determine the existence of
limits of the form 
\begin{equation*}
\lim_{(x,y)\rightarrow (a,b)}\frac{f(x,y)}{g(x,y)}
\end{equation*}%
where $f$ and $g$ are real analytic functions (typically, polynomials or
trigonometric and exponential functions) defined in an open disk centered at
a point $(a,b)$ in $\mathbb{R}^{2}$. The standard strategy for solving this
problem consists in studying the existence of the limit along various simple
trajectories, such as straight lines, quadrics, cubics, etc., with the hope
that either one of them fails to exist or two of them differ. If they all
coincide, then one tries some other ad hoc trajectories. If all that fails,
one tries to prove its existence by some theoretical methods.

In this paper we develop a theoretical method which completely solves this
problem. An algorithm for polynomials based on this method is implemented,
which proves to be more powerful than other existing routines. 

An application of Weierstrass' Preparation Theorem allows to reduce the
problem to the case where $f$ and $g$ are monic polynomial functions in the
variable $y,$ whose coefficients are real series in the variable $x.$ Next,
a \textit{discriminant real curve} is constructed using Lagrange Multipliers
with the property that the limit exists, if and only if, it exists along
this curve. Then Hensel's lemma, some Galois theory, and the theory of
Puiseaux series are used to parametrize the various branches of the
discriminant curve and select the \textit{real ones. }All these steps are
done in a constructive manner making it possible to implement this method in
an algorithmic way. In this article an algorithm was implemented for
polynomial functions.

\section{Theory}

\subsection{Reduction to the case where $f$ and $g$ are polynomials}

After a translation, we may assume that $(a,b)$ is the origin. 

Let us denote by $S=\mathbb{R}\left\{ x,y\right\} $ the ring of power series
in the variables $x,$ $y$ with real coefficients having positive radius of
convergence around the origin. If $h(x,y)$ belongs to $S$, the order of $h$
in the variable $y$ is defined to be the smallest integer $r$ such that $%
\overline{h}(y)=h(0,y)$ has the form

\begin{equation*}
\overline{h}(y)=\alpha _{r}y^{r}+\alpha _{r+1}y^{r+1}+\cdots ,\text{ with }%
\alpha _{r}\neq 0.
\end{equation*}%
(If $\overline{h}(y)=0$ the order is defined to be $+\infty .$)

It is not difficult to show that given $h_{1},\ldots ,h_{n}$ in $S-\{0\}$
there exists an integer $v\geq 1$ such that after a change of coordinates of
the form $x^{\prime }=x+y^{v},y^{\prime }=y$, each series $h_{i}^{\prime
}(x^{\prime },y^{\prime })=h_{i}(x+y^{v},y)$ is of finite order in the
variable $y^{\prime }$ \cite{Greuel}. The essential tool for the reduction
is the following lemma.

\begin{lemma}[Weierstrass]
Let $h$ be an element of $S=\mathbb{R}\left\{ x,y\right\} $ of order $d$ in $%
y.$ Then there exists a unique unit $u(x,y)\in S$ and unique real series $%
a_{1}(x),\ldots ,a_{d}(x)$ with positive radii of convergence such that $%
h(x,y)=u(x,y)(y^{d}+a_{1}(x)y^{d-1}+\cdots +a_{d}(x))$ \cite{Greuel}.
\end{lemma}

Since the existence of the limit and its value is obviously independent of
the particular choice of local coordinates, we may assume that $%
f(x,y)=u(x,y)f_{1}(x,y)$ and $g(x,y)=v(x,y)g_{1}(x,y)$, where $u\left(
x,y\right) $ and $v(x,y)$ are units and 
\begin{eqnarray*}
f_{1} &=&y^{d}+a_{1}(x)y^{d-1}+\cdots +a_{d}(x) \\
g_{1} &=&y^{b}+c_{1}(x)y^{b-1}+\cdots +c_{b}(x)
\end{eqnarray*}%
are monic polynomials in $\mathbb{R}\left\{ x\right\} \left[ y\right] .$
Since units do not affect the existence of the limit, there is no loss of
generality in assuming also that $u$ and $v$ are equal to $1.$

\subsection{Discriminant variety for the limit}

The following proposition provides a necessary and sufficient condition for
the existence of the limit.

Let $f=y^{d}+a_{1}(x)y^{d-1}+\cdots +a_{d}(x)$ and $g=y^{b}+c_{1}(x)y^{b-1}+%
\cdots +c_{b}(x)$ be monic polynomials in $\mathbb{R}\left\{ x\right\} \left[
y\right] ,$ and $D=D_{\rho }(0)\subset \mathbb{R}^{2}$ a closed disk
centered at the origin with radius $\rho >0$, such that each $a_{i}(x),$ $%
c_{j}(x)$ is convergent in $D.$ Let us denote by $h^{\prime }$ the
polynomial $y\partial q/\partial x-x\partial q/\partial y,$ where $q$
denotes the quotient $q=f/g,$ and by $h^{\prime \prime }$ (the numerator of $%
h^{\prime }$) the polynomial 
\begin{equation}
h^{\prime \prime }=y\left( g\frac{\partial f}{\partial x}-f\frac{\partial g}{%
\partial x}\right) -x\left( g\frac{\partial f}{\partial y}-f\frac{\partial g%
}{\partial y}\right) .  \label{1}
\end{equation}%
Let $X$ be the variety cut by $h^{\prime }$ in the puncture disk, i.e., 
\begin{equation*}
X=\left\{ (x,y)\in D:(x,y)\neq (0,0)\text{ and }h^{\prime \prime
}(x,y)=0\right\} .
\end{equation*}%
With this notation we have the following proposition.

\begin{proposition}
Let $q(x,y)=f(x,y)/g(x,y)$. The limit 
\begin{equation*}
\lim_{(x,y)\rightarrow (a,b)}q(x,y)
\end{equation*}%
exists and equals $L\in \mathbb{R}$, if and only if for every $\epsilon >0$
there is $0<\delta <\rho $ such that for every $(x,y)\in X\cap D_{\delta },$
the inequality $\left\vert q(x,y)-L\right\vert <\epsilon $ holds.
\end{proposition}

\begin{proof}
The method of Lagrange multipliers applied to the function $q(x,y),$ subject
to the condition $x^{2}+y^{2}=r^{2}$ where $0<r<\rho ,$ says that the
extreme values taken by $q(x,y)$ on each circle $C_{r}(0),$ centered at the
origin and having radius $r,$ occur among those points $(x,y)$ of $C_{r}(0)$
for which the vectors $\left( \partial q/\partial x,\partial q/\partial
y\right) $ and $(x,y)$ are parallel, which amounts to $y\partial q/\partial
x-x\partial q/\partial y=0.$ Let us assume that given $\epsilon >0$ there
exists $0<\delta <\rho $ such that for every $(x,y)\in X\cap D_{\delta }$
the inequality $\left\vert q(x,y)-L\right\vert <\epsilon $ holds. Let $%
(x,y)\in D_{\delta }$ and $r=\sqrt{x^{2}+y^{2}}.$ If $t_{1}(r),t_{2}(r)\in
C_{r}(0)$ are such that $q(t_{1}(r))=\min_{t\in C_{r}(0)}q(t),$ and $%
q(t_{2}(r))=\max_{t\in C_{r}(0)}q(t),$ then 
\begin{equation*}
q(t_{1}(r))-L\leq q(x,y)-L\leq q(t_{2}(r))-L,
\end{equation*}%
for every $(x,y)\in C_{r}(0).$ Since $t_{1}(r),t_{2}(r)\in X\cap D_{\delta }$
we have 
\begin{equation*}
-\epsilon <q(t_{1}(r))-L\text{ \ and }q(t_{2}(r))-L<\epsilon \text{\ }
\end{equation*}%
and therefore $\left\vert q(x,y)-L\right\vert <\epsilon ,$ for all $(x,y)\in
D_{\delta }.$

The "only if" part is immediate from the definition of limit.
\end{proof}

\subsection{Hensel's Lemma}

Let us fix an integer $n\geq 0.$ A linear change of coordinates of the form $%
f_{1}(x,y)=f(x+ny,-nx+y),$ and $g_{1}(x,y)=g(x+ny,-nx+y)$ does not alter the
limit of the quotient as $(x,y)$ approaches the origin. It is easy to see
that this change of coordinate transforms (\ref{1}) into a monic polynomial
multiplied by a nonzero constant. By the chain rule we have that $h^{\prime
\prime }(x+ny,-nx+y)$ is equal to 
\begin{equation*}
(-nx+y)[g_{1}(\partial f/\partial x)_{1}-f_{1}(\partial g/\partial
x)_{1}]-(x+ny)[g_{1}(\partial f/\partial y)_{1}-f_{1}(\partial g/\partial
y)_{1}],
\end{equation*}%
where $(\partial f/\partial x)_{1}(x,y)=\partial f/\partial x(x+ny,-nx+y),$
and similarly with $(\partial f/\partial y)_{1},$ $(\partial g/\partial
x)_{1},(\partial g/\partial y)_{1}.$ We will denote $h^{\prime \prime
}(x+ny,-nx+y)$ simply by $h(x,y).$

Our next goal is to parameterize the curve $h(x,y)=0.$ For this purpose we
will use Hensel's Lemma (\cite{Eisenbud}). Let us denote by $k$ a arbitrary
field, by $R$ the ring of formal power series in the variable $x,$ with
coefficients in $k,$ $R=k[[x]]$, and by $k((x))$ its field of fractions. $R$
is a local ring $(R,m)$ whose maximal ideal is $m=(x).$ For each $h(x,y)\in
R[y]$ monic in the variable $y,$ let us denote by $\overline{h}$ its
reduction modulo $m$, i.e., $\overline{h}=h(0,y).$

\begin{lemma}[Hensel's Lemma]
Let $F(x,y)$ be an element of $R[y]$ monic in $y,$ and let us assume that $%
\overline{F}=gh$ is a factorization in $k[y]$ whose factors are relatively
prime, and of degrees $r$ and $s.$ Then there exist unique $G$ and $H$ in $%
R[y]$ with degrees $r$ and $s,$ respectively, such that:

\begin{enumerate}
\item $\overline{G}=g$ and $\overline{H}=h$

\item $F=GH$
\end{enumerate}
\end{lemma}

In order to construct a parameterization of $h(x,y)=0,$ Puiseaux series are
used which we review next.

\subsection{Puiseaux Series}

Let us denote by $L$ the quotient field of fractions of $R=\mathbb{C}[[x]],$
which consists of Laurent series. Let $\overline{L}$ be an algebraic closure
of $L.$ For each positive integer $n$ we will denote by $x^{1/n}$ a fixed $n$%
-th root of $x$ in $\overline{L}.$ It is clear that the $n$-th roots of $x$
are%
\begin{equation*}
\theta x^{1/n},\theta ^{2}x^{1/n},...,\theta ^{n-1}x^{1/n},x^{1/n}
\end{equation*}%
where $\theta $ is any primitive $n$-th root of unity. It is easy to see
that the polynomial $t^{n}-x$ is irreducible in $L[t]$ and therefore $%
L\subset L(x^{1/n})$ is an extension of degree $n$. Consider the directed
system consisting of the positive natural numbers (partially) ordered by
divisibility, i.e., $n\leq m$ if and only if $n|m$. The direct limit $%
\lim_{\longrightarrow \mathbb{N}}L(x^{1/n})$ will be denoted by $L^{\ast }$.
This limit can be identified with the field $\cup _{n}L(x^{1/n})\subset 
\overline{L}$. Each element $\sigma $ of $L^{\ast }$ can therefore be
written in the form $\sigma =\sum c_{k}x^{q_{k}}$, with $c_{k}\in \mathbb{C}$%
, and exponents $q_{k}\in \mathbb{Q}$ such that:

\begin{enumerate}
\item $q_{1}<\cdots <q_{r}<\cdots $

\item There is an integer $b$ so that each exponent can be written as $%
q_{i}=a_{i}/b$, for some integer $a_{i}$.
\end{enumerate}

The least exponent in the expression for $\sigma ,$ $q_{1},$ is called \emph{%
the order of} $\sigma $. It is a well known theorem that given a monic
polynomial 
\begin{equation*}
h(x,y)=y^{d}+h_{1}(x)y^{d-1}+\cdots +h_{d}(x)
\end{equation*}%
in $R[y]$, there is an integer $N>0$ such that $h$ can be factored
completely in $L(x^{1/N})\subset L^{\ast }$ as 
\begin{equation}
h(x,y)=(y-\sigma _{1}(x^{1/N}))\cdots (y-\sigma _{d}(x^{1/N})),  \tag{1}
\label{factorizacion}
\end{equation}%
where each $\sigma _{i}(t)$ is an element of $\mathbb{C}[[t]]$, i.e. a
formal power series. Moreover, it can be seen that this series has positive
radius of convergence and therefore defines a holomorphic function (cf. \cite%
{Greuel}). Using this result, it is possible to \emph{parameterize} the
curve 
\begin{equation*}
X=\{(x,y)\in \mathbb{C}^{2}:(x,y)\neq (0,0)\ \text{and}\ h(x,y)=0\}.
\end{equation*}%
The proof of the existence of (\ref{factorizacion}) can be done
constructively using Hensel's Lemma (Lemma \ref{lema}) making it possible to
determine which one of the series $\sigma _{i}(t)$ has only real
coefficients. Such series will be called throughout, a \emph{real series}.
This in turn allows us to parameterize each one of the trajectories in $%
X\cap \mathbb{R}^{2}$ which go through the origin. It will be shown that
these are the only ones that are relevant, since for a holomorphic function
to be real valued on a real sequence approaching zero, it must have a series
expansion around the origin with only real coefficient.

The parameterization of the zeroes of $h$ can be done by observing that 
\begin{equation*}
h(x^{N},y)=(y-\sigma _{1}(x))\cdots (y-\sigma _{d}(x))
\end{equation*}%
and consequently $X=\{(x,y)\in \mathbb{C}^{2}:h(x,y)=0\}$ is the union of
the sets $X_{i}=\{(z^{N},\sigma _{i}(z)):z\in \mathbb{C}\}$. This allows us
to prove the following central result.

\begin{theorem}
Let $\sigma _{1}(z),\ldots ,\sigma _{l}(z),\ l\leq d$, be the real series in
the equation (\ref{1}) which go through de origin (i.e. $\sigma _{i}(0)=0$,
for $i=1,\ldots ,l)$ . Then the limit 
\begin{equation*}
\lim_{(x,y)\rightarrow (0,0)}\frac{f_{1}(x,y)}{g_{1}(x,y)}
\end{equation*}%
exists if and only if 
\begin{equation*}
\lim_{t\rightarrow 0}\frac{f_{1}(t^{N},\sigma _{i}(t))}{g_{1}(t^{N},\sigma
_{i}(t))}=L_{i}
\end{equation*}%
exists, for $i=1,\ldots ,l$, and $L_{1}=\cdots =L_{l}$.
\end{theorem}

\subsection{Newton's automorphism}

For each rational number $q\neq 0$ there exists a homomorphism $\alpha
_{q}:L^{\ast }\rightarrow L^{\ast }$, which sends $x$ to $x^{q}$ and fixes
the subfield $\mathbb{C}$. This homomorphism is constructed by first
defining a homomorphism from $\mathbb{C}[x]$ into $L^{\ast }$ which sends $x$
to $x^{q}$, then extending it to $\mathbb{C}[[x]]$, and then to the field of
fractions $\mathbb{C}((x))$. Since $L^{\ast }$ is an algebraic extension of $%
L$, this homomorphism extends to a homomorphism $\alpha _{q}$ from $L^{\ast
} $ to $\overline{L}$. It is clear that the image $\alpha _{q}$ lies inside $%
L^{\ast }$, and therefore one can regard $\alpha _{q}$ as an endomorphism of 
$L^{\ast }$. For $p\neq 0$ rational, let $\beta _{q,p}:L^{\ast
}[y]\rightarrow L^{\ast }[y]$ be the extension of $\alpha _{q}$ obtained
sending $y$ to $yx^{p}$. It is clear that $\beta _{q,p}$ is invertible and
its inverse is $\beta _{1/q,-p/q}$.

With these preliminaries we can now state the following fundamental theorem (%
\textrm{\cite{MOND}}). Even though this result is well known in the
literature, we provide a "constructive" proof, since it is the very heart of
the procedure \textit{sus }in the algorithm \textit{limite,} whose code we
give at the end.

\begin{theorem}
\label{Teoremilla} Every polynomial $h=y^{d}+h_{1}(x)y^{d-1}+\cdots
+h_{d}(x) $ with coefficients $h_{i}(x)$ in $L^{\ast }$ can be factored into
linear factors $h=(y-\sigma _{1})\cdots (y-\sigma _{d})$, with $\sigma
_{i}\in L^{\ast }$. Even more, if each $h_{i}(x)$ belongs to $\mathbb{C}%
[[x]] $, then there exists a positive integer $n$ such that all $\sigma
_{i}\in \mathbb{C}[[x^{1/n}]]$.
\end{theorem}

\begin{proof}
The proof proceeds by induction on the degree of $h$. The homomorphism $\phi
:L^{\ast }[y]\rightarrow L^{\ast }[y]$ sending $y$ to $y-(h_{1}(x)/d),$ and
fixing every element of $L^{\ast }$, is invertible and its inverse is the
homomorphism that fixes each element of $L^{\ast }$ and sends $y$ to $%
y+(h_{1}(x)/d)$. A simple calculation shows that $\phi (h)$ is a polynomial
with coefficients in $L^{\ast }$ such that the coefficient of $y^{d-1}$ is
zero. Then 
\begin{equation*}
\phi (h)=y^{d}+b_{2}(x)y^{d-2}+\cdots +b_{d}(x).
\end{equation*}%
Let us denote by $u_{i}$ the order of $b_{i}(x)$, and let $2\leq r\leq d$ be
the least index for which $u_{r}/r=\min \{u_{i}/i:2\leq i\leq d\}$. Let us
define $\psi =\beta _{r,u_{r}}:L^{\ast }[y]\rightarrow L^{\ast }[y]$. $\psi $
sends $x$ to $x^{r}$, $y$ to $x^{u_{r}}y$ (and its inverse sends $x$ to $%
x^{1/r}$, and $y$ to $x^{-u_{r}/r}y$). Clearly, 
\begin{align}
\psi (\phi (h))& =x^{du_{r}}y^{d}+b_{2}(x^{r})x^{(d-2)u_{r}}y^{d-2}+\cdots
+b_{d}(x^{r}) \\
& =x^{du_{r}}(y^{d}+x^{-2u_{r}}b_{2}(x^{r})y^{d-2}+\cdots
+x^{-ku_{r}}b_{k}(x^{r})y^{d-k}+\cdots +x^{-du_{r}}b_{d}(x^{r})).
\end{align}%
The order of each term $x^{-ku_{r}}b_{k}(x^{r})$ of the polynomial inside
the parentheses is given by $-ku_{r}+ru_{k}\geq 0$ (since $u_{k}/k\geq
u_{r}/r$). Furthermore, the order of $x^{-ru_{r}}b_{r}(x^{r})$ is $0$.
Therefore, if 
\begin{equation*}
F^{\prime }=y^{d}+x^{-2u_{r}}b_{2}(x^{r})y^{d-2}+\cdots
+x^{-ku_{r}}b_{k}(x^{r})y^{d-k}+\cdots +x^{-du_{r}}b_{d}(x^{r}),
\end{equation*}%
by taking $N$ large enough, we have that $F=F^{\prime }(x^{N},y)\in \mathbb{C%
}[[x]][y]$, and $F$ admits a modulo $x$ reduction $f=\overline{F}\in \mathbb{%
C}[y]$ having at least two distinct roots. For if $f$ had a single root $c$,
then it is impossible that $c=0$ because the order of $%
x^{-ru_{r}}b_{r}(x^{r})$ is zero. And if $c\neq 0$ then $%
f=(y-c)^{d}=y^{d}-dcy^{d-1}+\cdots $ would have a nontrivial term in $%
y^{d-1} $, which is also impossible. Therefore, since $f\in \mathbb{C}[y]$,
one obtains $f=f_{1}f_{2}$, with $f_{1}$ and $f_{2}$ monic and of degrees
strictly smaller than $d$. Hensel's Lemma guarantees the existence of a
lifting $F=F_{1}F_{2}$ with $F_{1}$ and $F_{2}$ monic, and of degrees
strictly smaller than the degree of $F$. Consequently, $F^{\prime
}(x,y)=F_{1}^{\prime }F_{2}^{\prime }$ with $F_{i}^{\prime
}=F_{i}(x^{1/N},y) $. In conclusion, $\psi (\phi
(h))=x^{du_{r}}F_{1}^{\prime }F_{2}^{\prime }$. By the induction hypothesis
we know that $F_{1}^{\prime }F_{2}^{\prime }=\prod_{j=1}^{d}(y-\sigma _{j})$%
, with $\sigma _{j}\in L^{\ast }$ and therefore 
\begin{align}
\phi (h)& =x^{du_{r}/r}\psi ^{-1}(F_{1}^{\prime }F_{2}^{\prime }) \\
& =x^{du_{r}/r}\prod_{j=1}^{d}(x^{-u_{r}/r}y-\psi ^{-1}(\sigma _{j})) \\
& =\prod_{j=1}^{d}(y-x^{u_{r}/r}\psi ^{-1}(\sigma _{j}))
\end{align}%
is also a product of linear factors. This finishes the induction. The last
claim in the theorem also follows by induction on the degree of $h$. In
fact, if the coefficients of $h$ belong to $\mathbb{C}[[x]]$, then $%
F^{\prime }\in \mathbb{C}[[x]][y]$ and consequently $F_{1}^{\prime }$ and $%
F_{2}^{\prime }$ also belong to $\mathbb{C}[[x]][y]$. The induction
hypothesis guarantees the existence of a positive integer $m$ such that $%
F_{1}^{\prime }F_{2}^{\prime }=\prod_{j=1}^{d}(y-\sigma _{j})$, with $\sigma
_{j}\in \mathbb{C}[[x^{1/m}]]$. Furthermore, in the polynomial $\phi (h)$
all the $u_{k}$ are nonnegative and therefore $u_{r}/r\geq 0$. Consequently,
each element $x^{u_{r}/r}\psi ^{-1}(\sigma _{j})=x^{u_{r}/r}\sigma
_{j}(x^{1/r})$ belongs to $\mathbb{C}[[x^{1/n}]]$ with $n=rm$, and therefore
the desired factorization for $h$ is obtained.
\end{proof}

It easily follows from the last part of the proof of theorem \ref{Teoremilla}
shows that

\begin{remark}
\label{nota} If $h=y^{d}+h_{1}(x)y^{d-1}+\cdots +h_{d}(x)$ is a monic
polynomial with coefficients in $\mathbb{C}[[x]]$, then there exists a power 
$r>0$, and polynomials $g_{1}(x,y)$ and $g_{2}(x,y)$ in $\mathbb{C}[[x]][y]$%
, which are monic in the variable $y$ and of degrees $d_{1},d_{2}<d$, such
that $h(x^{r},y)=g_{1}(x,y)g_{2}(x,y)$.
\end{remark}

Theorem \ref{Teoremilla} admits the following refinement (\cite{MOND}).

\begin{theorem}
\label{Teom} Let $h=y^{d}+h_{1}(x)y^{d-1}+\cdots +h_{d}(x)$ be a monic
polynomial with coefficients in $\mathbb{C}[[x]]$ that is irreducible over $%
L=\mathbb{C}((x))[y]$. Then, if $\omega $ denotes a primitive $d^{th}$ root
of unity, there exists $\sigma (t)=\sum_{k=0}^{\infty }c_{k}t^{k}$ such that 
\begin{equation*}
h=(y-\sigma (\omega x^{1/d}))\cdots (y-\sigma (\omega ^{d}x^{1/d}))
\end{equation*}%
where $\sigma (\omega ^{r}x^{1/d})=\sum_{k=0}^{\infty }c_{k}(\omega
^{r}x^{1/d})^{k}$.
\end{theorem}

The following result allows to identify the real series in the factorization
given in (\ref{Teoremilla}).

\begin{lemma}
\label{lema} Let $F=y^{d}+b_{1}(x)y^{d-1}+\cdots +b_{d}(x)$ be a monic
polynomial of degree $d$ in the variable $y$ and whose coefficients are real
power series, i.e. $b_{i}(x)\in \mathbb{R}[[x]]$. Then $\overline{F}%
=(y-r)^{d}$, with $r\in \mathbb{R}$, if and only if its factorization in
Puiseaux series has the form 
\begin{equation}
F=\prod_{i=1}^{s}F_{i}  \label{forma}
\end{equation}%
where 
\begin{equation*}
F_{i}=(y-\sigma _{i}(x^{1/d_{i}}))(y-\sigma _{i}(\omega
_{i}x^{1/d_{i}}))\cdots (y-\sigma _{i}(\omega _{i}^{d_{i}-1}x^{1/d_{i}})),
\end{equation*}%
and each $\sigma _{i}(t)=\sum_{k=0}^{\infty }c_{k}t^{k}$ is a real series, $%
\omega _{i}$ is a primitive $d_{i}^{th}$ root of unity, $%
\sum_{i=1}^{s}d_{i}=d,$ and $\sigma _{i}(0)=r$.
\end{lemma}

\begin{proof}
Let us first consider the only if part of the equivalence. We proceed by
induction on $d$. If the degree of $F$ is $d=1$, then $F=y-b_{1}(x)$, with $%
b_{1}(x)\in \mathbb{R}[[x]]$, is a factorization in Puiseaux series and $%
b_{1}(0)=r$. If $d>1$, by remark \ref{nota}, there exists an integer $N>0$
such that $F(x^{N},y)=G(x,y)H(x,y)$ where $G$ and $H$ are polynomials which
are monic and of degrees $d_{1},d_{2}<d$ in the variable $y$. Thus 
\begin{equation*}
\overline{F}=F(0,y)=G(0,y)H(0,y).
\end{equation*}%
And this implies that 
\begin{equation*}
\overline{G}=(y-r)^{d_{1}},\ \text{and}\ \overline{H}=(y-r)^{d_{2}}.
\end{equation*}%
By Hensel's Lemma applied to the ring $\mathbb{R}[[x]][y]$, and in
particular, by its claim about uniqueness, one obtains that $G$ and $H$ are
polynomials with coefficients in $\mathbb{R}[[x]]$. By the induction
hypothesis $G$ and $H$ can be factored in the form (\ref{forma}), so $F$ can
also be factored in this way. Conversely, if (\ref{forma}) holds, then it
follows, by taking $N=d_{1}\ldots d_{s}$, that 
\begin{equation}
F(x^{N},y)=\prod_{i=1}^{s}(y-\sigma _{i}(x^{e_{i}}))(y-\sigma _{i}(\omega
_{i}x^{e_{i}}))\cdots (y-\sigma _{i}(\omega _{i}^{d_{i}-1}x^{e_{i}}))
\label{otra}
\end{equation}%
where $e_{i}=N/d_{i}$ and $\sigma _{i}(t)$ is a real series. By replacing $x$
by $0$ in (\ref{otra}) one obtains 
\begin{equation*}
\overline{F}=F(0,y)=\prod_{i=1}^{s}(y-\sigma _{i}(0))^{d_{i}}=(y-r)^{d}.
\end{equation*}
\end{proof}

Let now $F=y^{d}+b_{1}(x)y^{d-1}+\cdots +b_{d}(x)$ be a monic polynomial
whose coefficients are in $\mathbb{R}[[x]]$ and let us denote its reduction
modulo $x$ by $f$. We can write%
\begin{equation*}
f=(y-r_{1})^{d_{1}^{\prime }}\cdots (y-r_{s})^{d_{s}^{\prime
}}(y-c_{1})^{d_{1}}(y-\overline{c_{1}})^{d_{1}}\cdots (y-c_{l})^{d_{l}}(y-%
\overline{c_{l}})^{d_{l}},
\end{equation*}%
where $r_{1},\ldots ,r_{s}$ are the real roots of $f$, and $c_{i},\overline{%
c_{i}}$ are the nonreal ones. Let us define $f_{i}(y)=(y-r_{i})^{d_{i}^{%
\prime }}$ and 
\begin{equation*}
g_{i}(y)=(y-c_{i})^{d_{i}}(y-\overline{c_{i}})^{d_{i}}=(y^{2}-\alpha
_{i}y+\beta _{i})^{d_{i}},
\end{equation*}%
with $\alpha _{i},\beta _{i}$ real. Hensel's Lemma provides us with a
lifting of the factorization $f_{1}\cdots f_{s}g_{1}\cdots g_{l}$, of the
form 
\begin{equation*}
F=F_{1}\cdots F_{s}G_{1}\cdots G_{l}
\end{equation*}%
i.e., $\overline{F_{i}}=f_{i}$ and $\overline{G_{i}}=g_{i}$. From the proof
of Hensel's Lemma it follows that each $F_{i}$ is a monic polynomial in the
variable $y$ with coefficients in $\mathbb{R}[[x]]$. Each $G_{i}$ admits a
factorization $\prod_{j=1}^{q_{i}}G_{ij}$ into irreducible factors in $%
\mathbb{C}[[x]][y]$, and, by Gauss' Lemma, also in $\mathbb{C}((x))[y]$ (see 
\cite{Lang}). Notice that if $G_{ij}$ has degree $e_{ij}$, then $%
\sum_{j=1}^{q_{i}}e_{ij}=2d_{i}$. By theorem \ref{Teom} each $G_{ij}$ can be
factored as 
\begin{equation*}
G_{ij}=\prod_{k=1}^{e_{ij}}(y-\sigma _{j}(\omega _{j}^{k}x^{1/e_{ij}})),
\end{equation*}%
where $\omega _{j}$ is an $e_{ij}^{th}$ primitive root of unity and $\sigma
_{j}(t)\in \mathbb{C}[[t]]$. If we let $e_{i}=e_{i1}\ldots e_{iq_{i}}$, it
is clear that 
\begin{equation*}
G_{i}(x^{e_{i}},y)=\prod_{j=1}^{q_{i}}\prod_{k=1}^{e_{ij}}(y-\sigma
_{j}(\omega _{j}^{k}x^{e_{i}/e_{ij}})),
\end{equation*}%
and as a consequence 
\begin{equation*}
G_{i}(0,y)=\prod_{j=1}^{q_{i}}(y-\sigma _{j}(0))^{e_{ij}}.
\end{equation*}%
It follows that $q_{i}=2$ and $\sigma _{1}(0)=c_{i}$ and $\sigma _{2}(0)=%
\overline{c}_{i}$ (or the other way round). Therefore none of the $\sigma
_{j}$ is a real series. By lemma \ref{lema}, $F_{k}=\prod_{h=1}^{m}F_{kh},$
and 
\begin{equation}
F_{kh}=(y-\sigma _{hk}(x^{1/d_{hk}}))(y-\sigma _{hk}(\omega
_{hk}x^{1/d_{hk}}))\cdots (y-\sigma _{hk}(\omega
_{hk}^{d_{hk}-1}x^{1/d_{hk}})),
\end{equation}%
with $\sigma _{hk}(t)$ being a real series. The following theorem summarizes
what has been achieved so far.

\begin{theorem}
Let $F=y^{d}+b_{1}(x)y^{d-1}+\cdots +b_{d}(x)$ be a polynomial that is monic
in the variable $y$ and whose coefficients lie in $\mathbb{R}[[x]]$, and let 
$f$ be its reduction modulo $x$. Then $f$ can be written as 
\begin{equation*}
f=(y-r_{1})^{d_{1}^{\prime }}\cdots (y-r_{s})^{d_{s}^{\prime
}}(y-c_{1})^{d_{1}}(y-\overline{c_{1}})^{d_{1}}\cdots (y-c_{l})^{d_{l}}(y-%
\overline{c_{l}})^{d_{l}}.
\end{equation*}%
Let%
\begin{equation*}
f_{i}(y)=(y-r_{i})^{d_{i}^{\prime }},\text{ }g_{i}(y)=(y-c_{i})^{d_{i}}(y-%
\overline{c_{i}})^{d_{i}}=(y^{2}-\alpha _{i}y+\beta _{i})^{d_{i}},
\end{equation*}%
with $r_{i},\alpha _{i},\beta _{i}$ real. Hensel's Lemma gives a lifting of
the factorization of $f=$ $f_{1}\cdots f_{s}g_{1}\cdots g_{l}$ 
\begin{equation*}
F=F_{1}\cdots F_{s}G_{1}\cdots G_{l}
\end{equation*}%
with $\overline{F_{i}}=f_{i}$ and $\overline{G_{i}}=g_{i}$. Then, in the
Puiseaux series factorization of $F$ the only real series occur in the
decomposition into linear factors of the $F_{i}$.
\end{theorem}

Thus if $X$ denotes the curve in $\mathbb{C}^{2}$ formed by the zeroes of $h$
then 
\begin{equation*}
X\cap \mathbb{R}^{2}=\cup _{i=1}^{s}\{(t^{d},\sigma _{i}(t^{m_{i}})):t\in
\lbrack a,b]\subset \mathbb{R}\}.
\end{equation*}

\section{Algorithm for the computation of limits}

As an application of the results established in the previous section, we
present in this section an algorithm implemented in \emph{Maple 12} for the
determination of limits of the form 
\begin{equation*}
\underset{(x,y)\rightarrow (0,0)}{\lim }\frac{f(x,y)}{g(x,y)}
\end{equation*}
where $f,g\in \mathbb{R}[x,y]$ and $g$ has an isolated zero at $(0,0)$. The
complete routine is called \emph{limite} and comprises several subroutines
that are documented next.

The routine \textit{repeticion} takes as input a list $L$ and forms a list
of lists, each one formed by each element of $L$, repeated as many times as
it appears in $L$. For instance, if $L=[1,2,1,3,1,2,4,3]$ then \textit{%
repeticion} produces the list $[[1,1,1],[2,2],[3,3],[4]]$. The \textit{Maple
12} code for this routine is as follows. \medskip

$>$ repeticion:=proc(L)

$>$ local S,j,H,i;

$>$ H:=convert(convert(L,set),list);

$>$ for i from 1 to nops(H) do S[i]:=[];

$>$ for j from 1 to nops(L)do

$>$ if L[j]=H[i] then S[i]:=[op(S[i]),H[i]]; end if;

$>$ end do; end do;

$>$ RETURN([seq(S[i], i=1..nops(H))]);

$>$ end proc:\medskip

The routine \textit{suprime} takes as input a list $L$ and eliminates its
redundancy. For instance, if $L=[1,a,1,3,a,b,c]$, then \emph{suprime}
returns $G=[1,a,3,b,c]$. The code for this routine is \medskip

$>$ suprime:=proc(L)

$>$ local S,i,G; G:=[L[1]];

$>$ for i from 2 to nops(L) do

$>$ S:=convert(evalf(G),set);

$>$ if (evalb(`in`(evalf(L[i]),S)) = false) then G:=[op(G),L[i]];

$>$ end if;

$>$ end do;

$>$ RETURN(G);

$>$ end proc:\medskip

The routine \textit{poli} takes as input a list $L$ whose elements are
complex numbers and constructs another list containing a polynomial of the
form $(y-r)^{d}$ for each real $r$ that appears exactly $d$ times in $L$,
and a polynomial of the form $((y-z)(y-\overline{z}))^{d}$ for each nonreal $%
z$ appearing together with its conjugate $\overline{z}$ exactly $d$ times in 
$L$. For example, if $L=[1,2-i,1,2+i]$ then \textit{poli} returns $%
[(y-1)^{2},(y-(2-i))(y-(2+i))]$. Its code is\medskip

$>$ poli:=proc(L)

$>$ local g,H,i;

$>$ H:=repeticion(L);

$>$ for i from 1 to nops(H) do

$>$ if Im(evalf(H[i][1]))=0 then g[i]:=(y-H[i][1])\symbol{94}nops(H[i]); else

$>$ g[i]:=((y-H[i][1])*(y-conjugate(H[i][1])))\symbol{94}nops(H[i]); end if;
end do;

$>$ suprime([seq(g[i],i=1..nops(H))]);

$>$ end proc:\medskip

The routine \textit{mochar} takes a polynomial $f(x,y)=a_{0}(x)y^{d}+\cdots
+a_{k}(x)y^{k}+\cdots +a_{d}(x)$ and eliminates from each coefficient those
powers of $x$ which are larger than $n$, i.e. it calculates $f(x,y)$ modulo $%
x^{n+1}$. The code for this routine is\medskip

$>$ mochar:=proc(f,n)

$>$ local c,d,i,g;

$>$ d:=degree(f,y); g:=0;

$>$ for i from 0 to d do

$>$ c[d-i]:=mtaylor(coeff(f,y,d-i), [x], n+1);

$>$ g:=g+c[d-i]*y\symbol{94}(d-i);

$>$ end do;

$>$ collect(g,y);

$>$ end proc:\medskip

The routine \textit{monico} takes a polynomial $f$ in the variable $y$ and
divides it by the coefficient of the highest power of $y$. Its code
is\medskip

$>$ monico:=proc(f)

$>$ local c,d;

$>$ d:=degree(f,y);

$>$ c:=coeff(f,y,d);

$>$ collect(expand(1/c*f),y);

$>$ end proc:\medskip

The routine \textit{Hensel} has four entries. The first entry is a
polynomial $F(x,y)$ en $\mathbb{C}[x][y]$, which is monic in $y$. The second
and third entries are polynomials $g(x)$, $h(y)$ such that $F(0,y)=g(y)h(y)$%
. The fourth entry is an integer $n$. \textit{Hensel} calculates polynomials 
$G(x,y)$ and $H(x,y)$ such that $F=GH$ modulo $x^{n+1}$. The code for this
routine is\medskip

$>$ Hensel:=proc(poly,gg,hh,n)

$>$ local L,H,G,i,l,f,g,h,t;

$>$ f[0]:=coeff(poly,x,0);

$>$ g[0]:=gg;

$>$ h[0]:=hh;

$>$ f[1]:=coeff(poly,x,1);

$>$ gcdex(g[0],h[0],f[1],y,'s','t'); h[1]:=s;g[1]:=t;

$>$ for i from 2 to n do

$>$ f[i]:=coeff(poly,x,i);l[i]:=f[i]-sum(g[j]*h[i-j],'j'=1..i-1);

$>$ gcdex(g[0],h[0],l[i],y,'s','t');

$>$ h[i]:=s;g[i]:=t;

$>$ end do;

$>$ H:=sum(h[j]*x\symbol{94}(j),'j'=0..n);

$>$ G:=sum(g[j]*x\symbol{94}(j),'j'=0..n);

$>$ L:=[mochar(G,n), mochar(H,n)]; RETURN(L[1],L[2]);

$>$ end proc:\medskip

The routine \textit{henselgen} takes as entry a polynomial $f(x,y)$ which is
monic in $y$, a list $L$ of polynomials $f_{i}(y)$, all monic in $y$,
pairwise relatively prime and such that $f(0,y)=f_{1}(y)\cdots f_{r}(y)$,
and an integer $n>0$. This procedure returns polynomials $%
G_{1}(x,y),...,G_{r}(x,y)$ such that $f(x,y)=G_{1}(x,y)\cdots G_{r}(x,y)$
modulo $x^{n+1}$, and $G_{i}(0,y)=f_{i}(y)$. The Maple 12 code for this
routine is \medskip

$>$ henselgen:=proc(f,L,n)

$>$ local P,G,H,T,S,i; S:=convert(L,set); H:=[];

$>$ for i from 1 to nops(L) do

$>$ T:= convert(S minus \{L[i]\},list);

$>$ P[i]:=expand(product(T[k],k=1..nops(T)));

$>$ G[i]:=Hensel(expand(f),expand(L[i]),P[i],n)[1]; H:=[op(H),G[i]];

$>$ end do;

$>$ RETURN(H); end proc:\medskip

The routine \textit{orden} receives a polynomial \textit{``poly"} 
\begin{equation*}
f(x,y)=y^{d}+b_{1}(x)y^{d-1}+\cdots +b_{d}(x),
\end{equation*}
with each $b_{i}(x)$ being a Laurent polynomial in $x$. It computes $u=\min
\{u[i]:i=1,...,d\}$, with $u[i]$ being the order of $b_{i}(x)$ (i.e. the
degree of the least degree term in $b_{i}(x)$), and returns $[r,u[r]]$, with 
$r$ the smallest index $i$ such that $u[r]/r=\min\{u[i]/i\}$. Its code is
\medskip

$>$ orden:=proc(poly)

$>$ local u,i,b,r,m,d;

$>$ d:=degree(poly, y);

$>$ b[1]:=coeff(poly,y,d-1); u[1]:=ldegree(b[1]); r:=1; m:=u[1]/r;

$>$ for i from 2 to d do

$>$ b[i]:=coeff(poly,y,d-i); u[i]:=ldegree(b[i]);

$>$ if (u[i]/i $<$ m ) then

$>$ m:= u[i]/i; r:=i;

$>$ end if;

$>$ end do; RETURN([r,u[r]]);

$>$ end proc:\medskip

The routine \textit{sus} receives a polynomial \textit{``poly"}, 
\begin{equation*}
poly(x,y)=y^{d}+b_{1}(x)y^{d-1}+\cdots +b_{d}(x),
\end{equation*}
whose coefficients are Laurent series in $x$ and computes:

\begin{enumerate}
\item A polynomial $g(x,y)$ in $\mathbb{C}((x))^{\ast }[y]$ satisfying $\psi
\phi (f(x,y))=x^{du[r]}g(x,y)$, (for notation see \ref{Teoremilla}) where $f$
denotes the polynomial that is obtained from \textit{poly} after performing
the substitution $y=y-b_{1}(x)/d,$ i.e. 
\begin{equation*}
f(x,y)=poly(x,y-b_{1}(x)/d)=y^{d}+c_{2}(x)y^{d-2}+\cdots +c_{d}(x).
\end{equation*}%
Then $r$ is equal to the value of the index $i$ where the minimum value of $%
\{u[i]/i:u[i]=$ degree of $c_{i}(x)\}$, is attained, and $u[r]$ is equal to
the order of $c_{r}(x)$. where 
\begin{equation*}
\psi \phi :\mathbb{C}((x))^{\ast }[y]\rightarrow \mathbb{C}((x))^{\ast }[y],
\end{equation*}%
denotes the automorphism obtained by composing the map $\phi
:f(x,y)\longmapsto $ $f(x,y-b_{1}(x)/d)$, with $\psi :f(x,y)\longmapsto
f(x^{r},yx^{u[r]}).$

\item \textit{sus} returns the triple $[g,r,u[r]]$, with $%
g(x,y)=x^{-du[r]}\psi \phi (f/x,y)$.\medskip
\end{enumerate}

\textit{Warning:}

1. If $u[r]=\infty $ ( which happens in case $f(x,y)=y^{d}$, or equivalently
if $poly(x,y)=(y-b_{1}(x))^{d}$), then \textit{sus} returns the triple $%
[(y-b_{1}(x))^{d},1,\infty ]$.

2. Notice that $r$ and $u[r]$ correspond to the polynomial $f(x,y)$ obtained
from \textit{``poly"} after performing the linear substitution that
eliminates the term of degree ${d-1}$ in $y$, but not to the polynomial 
\textit{``poly"} itself.\medskip

The code for the routine \emph{sus} is \medskip

$>$ sus:=proc(poly)

$>$ local g,q,f,b,d,u,r;

$>$ d:=degree(poly,y);

$>$ b[1]:= coeff(poly, y, d-1);

$>$ f:=collect(simplify(subs( y=y-b[1]/d, poly)),y);

$>$ r:=orden(f)[1]; u[r]:=orden(f)[2]; if u[r]=infinity then
RETURN([factor(poly),r,u[r]]);end if;

$>$ q:=collect( expand(x\symbol{94}(-d*u[r])*subs( \{x=x\symbol{94}r, y=y*x%
\symbol{94}u[r]\}, f )), y); RETURN([q,r,u[r]]);

$>$ end proc:\medskip

The routine \textit{invsus} takes as input an integer $d$, a polynomial $%
b=b(x)$, integers $r,u$ and a polynomial $g(x,y)=y^{d}+c_{1}(x)y^{d-1}+%
\cdots +c_{d}(x)$. If $\psi ^{-1}$ is the automorphism determined by $%
x\longmapsto x^{1/r}$, $y\longmapsto yx^{-u/r}$, and $\phi $ is the
isomorphism defined as the linear substitution $y\longmapsto y+b(x)/d$, then 
\textit{\ invsus} computes $x^{du/r}\phi ^{-1}\psi ^{-1}(g(x,y))$.\medskip

\textit{Observation:}\medskip

If $d=$grado($f(x,y)$), $b=b[1]$ the coefficient of $y^{d-1}$ in $f(x,y)$, $%
r $ the smallest index $i$ such that $u[r]/r=\min\{u[i]/i\}$, where each $%
u[i]$ is the degree of the coefficient of $y^{d-i}$ in the polynomial $%
f(x,y-b/d)$ and $u=u[r]$ the order of its $r^{th}$ coefficient $c_{r}(x)$,
(quantities associated to $f(x,y)$ in the procedure \emph{sus}, and $g(x,y)$
is a polynomial monic and of degree $d$ in $y$, resulting from the previous
procedure, i.e. such that $g=$\textit{sus}$(f(x,y))[1]$, and therefore 
\begin{equation*}
g(x,y)=x^{-du[r]}\psi \phi (f/x,y)).
\end{equation*}%
Then \textit{invsus} returns as a result $f(x,y)$, because 
\begin{eqnarray*}
x^{du[r]/r}\phi ^{-1}\psi ^{-1}(x^{-du[r]}\psi \phi (f/x,y)) &=& \\
x^{du[r]/r}x^{-du[r]/r}\phi ^{-1}\psi ^{-1}(\psi \phi (f/x,y)) &=&f(x,y).
\end{eqnarray*}

Here is the code for \emph{sus}.

$>$ invsus:=proc(d,b,r,u,g)

$>$ local D,t;D:=degree(g,y);

$>$ t:=root(x,r,symbolic);

$>$ simplify(t\symbol{94}(D*u)*subs(\{x=t, y=y*t\symbol{94}%
(-u)\},g),symbolic);

$>$ collect(simplify(subs(y=y+b/d,\%)),y);

$>$ RETURN(\%);

$>$ end proc:\medskip

The routine \textit{reduccion} receives a polynomial $f(x,y)$ which is monic
in $y$, and an integer $n>0$. If $f$ is linear or has the form $%
(y-b_{1}(x))^{d}$, the algorithm returns the triple $[1,1,f(x,y)]$.
Otherwise, it sets $g(x,y)$ $=$\textit{sus}$(f)$ (hence $\psi \phi
(f/x,y)=x^{du[r]}g(x,y)$) and tries to verify whether $g(0,y)$ has at least
one real root. If this is not the case, the algorithm returns $[1,1,f(x,y)]$%
. Otherwise, it factors $g(0,y)$ in terms of the form $g(0,y)$ $%
=h_{1}(y)\cdots h_{m}(y)$, where each $h_{i}(y)$ has the form $%
(y-r_{i})^{d_{i}}$, with $r_{i}$ real or of the form $h_{i}(y)=[(y-z)(y-%
\overline{z})]^{d_{i}}$, with $z$ nonreal, lifts the factorization using 
\textit{henselgen} modulo $x^{n+1}$, and applies \textit{invsus} in order to
obtain an integer $r>0$ such that $f(x^{r},y)=q_{1}(x,y)\cdots q_{m}(x,y)$.
The procedure returns the list $[r,[q_{1}(x,y),...,q_{m}(x,y)]]$. Here is
the code for this routine. \medskip

$>$ reduccion:=proc(f,n)

$>$ local i,L,S,d,h,q,r,b,g,ra,H,u;

$>$ d:=degree(f,y); if d=1 then RETURN([1,[1,f]]); end if;
b[1]:=coeff(f,y,d-1);r:=sus(f)[2];u:=sus(f)[3];if u=infinity then
RETURN([1,[1,factor(f)]]); end if;

$>$ g:=unapply(sus(f)[1],x,y);

$>$ if [fsolve(g(0,y),y)]=[] then RETURN([1,[1,f]]); end if;

$>$ S:=[solve(g(0,y),y)];

$>$ L:=poli(S);

$>$ H:=henselgen(g(x,y),L,n);

$>$ for i from 1 to nops(H) do

$>$ h[i]:=unapply(invsus(d,b[1],r,u,H[i]),x,y);

$>$ q[i]:=mochar(simplify(h[i](x\symbol{94}r,y),symbolic),n);

$>$ end do;

$>$ RETURN([r,[seq(q[i],i=1..nops(H))]]);

$>$ end proc:\medskip

The routine \textit{fiscal} checks the list $L$ and returns the empty list $%
[ \ ]$, if and only if all elements in $L$ are powers of linear polynomials
(i.e. each $L_{i}(x,y)$ in the list $L$ has the form $(y-b_{1}(x))^{d},d\geq
1$) or each $L_{i}(x,y)$ satisfies that \textit{sus}$(L_{i}(x,y))$ has only
nonreal roots. Otherwise, it returns the same list $L$. Here is the code for
this routine. \medskip

$>$ fiscal:=proc(L)

$>$ local ra,g,i;

$>$ for i from 1 to nops(L) do

$>$ g:=unapply(sus(L[i])[1],x,y);

$>$ ra:=[fsolve(g(0,y),y)];

$>$ if (degree(L[i],y)$>$1 and ra$<>$[] and sus(L[i])[3]$<>$infinity ) then
RETURN(L);

$>$ end if; end do; RETURN([]);

$>$ end proc:\medskip

The routine \textit{cambio} receives a list $%
L=[r,[r_{1},r_{2},...r_{i-1},r_{i},r_{i+1},...,r_{n}]]$, two integers $i$, $%
m $, and another integer $b$. The procedure returns a new list $%
L=[br,[br_{1},br_{2},...br_{i-1},r_{i},r_{i},...,r_{i},br_{i+1},...,br_{n}]$
(from the $i^{th}$ position on it puts $r_{i}$ repeated $m$ times). The code
for \emph{cambio} is \medskip

$>$ cambio:=proc(L,i,b,m)

$>$ local n,S,h; n:=nops(L[2]);

$>$ if i=1 then S:=[ L[1]*b, [seq(L[2][1],j=1..m),seq(b*L[2][j],j=i+1..n) ]
]; end if;

$>$ if i =n then S:=[ L[1]*b,[ seq(b*L[2][j],j=1..n-1),seq(L[2][n],j=1..m)]
]; end if;

$>$ S:=[
L[1]*b,[seq(b*L[2][j],j=1..i-1),seq(L[2][i],j=1..m),seq(b*L[2][j],j=i+1..n)]];

$>$ S; end proc:\medskip

The routine \textit{factoriza} takes a polynomial $f(x,y)$ that is monic in $%
y$ and an integer $n$. The algorithm produces a list $L=[L1,L2]$ with
entries having the form $L1=[f_{1}(x,y),...,f_{n}(x,y)]$ and $%
L2=[r,[r_{1},...,r_{n}]]$ such that $r=r_{1}\cdots r_{n}$ and 
\begin{equation*}
f(x^{r},y)=f_{1}(x^{r_{1}},y)\cdots f_{n}(x^{r_{n}},y),
\end{equation*}%
such that each $f_{i}(x,y)$ admits no real reduction, i.e. $f_{i}(x,y)$ is
either of the form $(y-b_{1}(x))^{d}$, with $d\geq 1$, or if $g_{i}(x,y)=$%
\textit{sus}$(f_{i}(x,y))$, then $g_{i}(0,y)$ does not admit any real root.
In other words, the only real Puiseux series in the factorization of $f(x,y)$
are the ones given by $f_{1}(x^{r_{1}/r},y),...,f_{n}(x^{r_{n}/r},y)$, and
therefore if 
\begin{equation*}
f(x,y)=(y-\sigma _{1}(x^{1/a_{1}}))^{d_{1}}\cdots (y-\sigma
_{n}(x^{1/a_{n}}))^{d_{n}}(y-\alpha _{1}(x^{1/b_{1}}))^{e_{1}}\cdots
(y-\alpha _{m}(x^{1/b_{m}}))^{e_{m}}
\end{equation*}%
is the complete factorization of $f(x,y)$ in $\mathbb{C}((x))^{\ast }[y]$,
with $\sigma _{i}(t)=\sum_{k}p_{i_{k}}t^{k}$, a series with real coeficients 
$p_{i_{k}}$, and $\alpha _{j}(t)=\sum_{k}c_{j_{k}}t^{k}$, series having at
least one nonreal coefficient, then $f_{i}(x,y)=(y-\sigma
_{i}(x^{1/a_{i}}))^{d_{i}}$ modulo $x^{n+1}$, y $a_{i}=r_{i}/r$. The code
for \emph{factoriza} is \medskip

$>$ factoriza:=proc(f,n)

$>$ local FF, RR, i, redu;

$>$ FF:=[f];RR:=[1,[1]];

$>$ while (fiscal(FF)$<>$[]) do

$>$ for i from 1 to nops(FF) do

$>$ if reduccion(FF[i],n)[2][1]$<>$1 then

$>$ redu:=reduccion(FF[i],n);RR:=cambio(RR,i,redu[1],nops(redu[2]));

$>$ if i $<$ nops(FF) then
FF:=[op(FF[1..i-1]),op(redu[2]),op(FF[i+1..nops(FF)])];

$>$ else FF:=[op(FF[1..i-1]),op(redu[2])]; end if;

$>$ else end if;end do;

$>$ end do;

$>$ RETURN([factor(FF),RR]);

$>$ end proc:\medskip

The routine \textit{rotacion} takes a polynomial $f(x,y)$ and looks for an
integer $n>0$ such that the substitution $x=x+ny$, $y=-nx+y$ makes it
quasi-monic. Afterwards it divides it by a nonzero constant making it monic.
The algorithm produces a $[g/c,n]$, where $g$ is obtained from $f$ via the
substitution, $c$ is the coefficient of the highest power of $y$ in the
polynomial $g$, and $n$ is the integer that makes this substitution work. If 
$f(x,y)$ is monic, the algorithm takes $n=0$ and hence it returns $f(x,y)$
again. The code for this routine is\medskip

$>$ rotacion:= proc(f)

$>$ local c,d,g,n; for n from 0 to infinity do

$>$ subs(\{x=x+n*y, y=-n*x+y\},f);

$>$ g:=collect(simplify(\%),y); d:=degree(g,y);

$>$ c:=coeff(g,y,d);

$>$ if degree(c,x)=0 then RETURN([g/c,n]); end if;

$>$ end do; end proc:\medskip

The routine \textit{revisor} takes as entry a polynomial $f(x,y)$ that is
monic in $y$. If $f(x,y)=g(x,y)^{n}$, then \textit{revisor} return a
polynomial $g(x,y)$, if $g(x,y)$ is linear in $y$, i.e. of the form $y-a(x)$%
. Otherwise, it returns the empty set. The code for this routine is \medskip

$>$ revisor:=proc(f)

$>$ local l;

$>$ l:= factors(f)[2];

$>$ if degree(l[1][1],y)=1 then RETURN(monico(l[1][1])); end if;

$>$ RETURN(); end proc:\medskip

\textit{limite} is the main routine and it is built out of the previous
ones. It receives as entry polynomials $f(x,y),g(x,y)$ (not necessarily
monic) and an integer $n>0$ . It computes the quotient $q=f/g$. Then it
computes 
\begin{equation*}
h=y(g(\partial f/\partial x)-f(\partial g/\partial x))-x(g(\partial
f/\partial y)-f(\partial g/\partial y)).
\end{equation*}%
Then it performs an appropriate rotation $h1=$\textit{\ rotacion}$(h)$, so
that $h1$ is monic in $y$, and factors $h1$ into irreducible polynomials $%
h1=b1\cdots bs$. Next if applies \textit{factoriza}, to an approximation of $%
n>0$, and returns for each factor $bi$, a list $L$ formed by two lists 
\begin{equation*}
L1=[[p_{1}(x,y),...,p_{n}(x,y)]\text{ y }L2=[r,[r_{1},...,r_{n}]]
\end{equation*}%
such that $r=r_{1}\cdots r_{n}$ y $bi(x^{r},y)=p_{1}(x^{r_{1}},y)\cdots
p_{n}(x^{r_{n}},y)$.\newline

After this, it computes the list $%
S=[p_{1}(x^{r_{1}},y),...,p_{n}(x^{r_{n}},y)]$. Each $%
p_{i}(x^{r_{i}},y)=(y-a_{i}(x^{r_{i}}))^{d_{i}}$, (which would correspond to
a real Puiseux series) or having the form $p_{i}=(u(x,y))^{d_{i}}$, where $%
u(x,y)$ is monic in $y$, and of degree larger than one (which would
correspond to a nonreal Puiseux series). Observe that $%
bi(x,y)=p_{1}(x^{r_{1}/r},y) \cdots p_{n}(x^{r_{n}/r},y)$.\newline

After this, \textit{revisor} is applied to $S$ and returns a list 
\begin{equation*}
H=[q_{i1}(x,y),q_{i2}(x,y),...,q_{ik}(x,y)],
\end{equation*}
where th $q_{ij}(x,y)=(y-a_{ij}(x^{r_{ij}}))$ are exactly those elements of $%
S$ , without the power $d_{ij}$, that correspond to real Puiseux series.
That is 
\begin{equation*}
q_{i1}(x,y)^{d_{i1}}=p_{i1}(x^{ri1},y),...,q_{ik}(x,y)^{d_{ik}}=p_{ik}(x^{r_{ik}},y).
\end{equation*}
In this way $%
q_{i1}(x,y)=y-a_{i1}(x^{r_{i1}}),...,q_{ik}(x,y)=y-a_{ik}(x^{r_{ik}})$.%
\newline

Then it constructs $f_{1}$ and $g_{1}$, the polynomial obtained by applying
the substitution $x=x+ny$, $y=-nx+y$ to $f$ and $g$ (using the same $n$
found for the rotation of $h$). After this it makes the list $%
T=[a_{i1}(x^{r_{i1}}),...,a_{ik}(x^{r_{ik}})]$. Then it computes $%
R=[a_{i1}(x^{r_{i1}/r}),...,a_{ik}(x^{r_{ik}/r})]$, and then it computes the
list $P$ of pairs

\begin{equation*}
P=[[f_{1}(x,a_{i1}(x^{r_{i1}/r})),g_{1}(x,a_{i1}(x^{r_{i1}/r}))],...,[f_{1}(x,a_{ik}(x^{r_{ik}/r})),g_{1}(x,a_{ik}(x^{r_{ik}/r}))]]
\end{equation*}%
excluding those trajectories not passing through the origin ( $%
a_{ik}(0^{r_{ij}/r})\neq 0$ ) and therefore the list of limits

\begin{equation*}
Q[i]=[\underset{x\rightarrow 0}{\lim }%
(f_{1}(x,a_{i1}(x^{r_{i1}/r}))/g_{1}(x,a_{i1}(x^{r_{i1}/r}))),...,\underset{%
x\rightarrow 0}{\lim }%
(f_{1}(x,a_{ik}(x^{r_{ik}/r}))/g_{1}(x,a_{ik}(x^{r_{ik}/r})))],
\end{equation*}%
for $i=1,...,s$.

The limit exists if all the values of the list $Q[i]$, for all $b_{i},$ $%
i=1,...,s$, are equal.\medskip

$>$ limite:=proc(f,g,n)

$>$ local F,G,a,b,Q,e,P,k,T,j,N,f1,g1,i,S,H,L,h1,h,q,R,t;

$>$ q:=f/g;

$>$ h:=
simplify(expand(y*(g*diff(f,x)-f*diff(g,x))-x*(g*diff(f,y)-f*diff(g,y))));
if h=0 then RETURN(q) end if;

$>$ h1:=rotacion(h);

$>$ F:=factors(h1[1])[2]; G:=[]; for b from 1 to nops(F) do
G:=[op(G),monico(F[b][1])]; end do;

$>$ for a from 1 to nops(G) do Q[a]:=[];

$>$ L:=convert(evalf(factoriza(G[a],n)),rational); S:=[];
t:=root(x,L[2][1],symbolic);

$>$ for i from 1 to nops(L[1]) do

$>$ S:=expand([op(S),subs(x=x\symbol{94}L[2][2][i],L[1][i])]); end do;

$>$ H:=map(revisor,S); N:=h1[2];

$>$ subs(\{x=x+N*y, y=-N*x+y\},f);f1:=collect(\%,y);subs(\{x=x+N*y,
y=-N*x+y\},g);

$>$ g1:=collect(\%,y); T:=[];

$>$ for j from 1 to nops(H) do

$>$ T:=[op(T),-coeff(H[j],y,0)/coeff(H[j],y,1)]; end do;
R:=subs(x=t,T);P:=[];

$>$ for k from 1 to nops(R) do if subs(x=0,R[k])=0 then

$>$ P:=[op(P),subs(\{y=R[k]\},[f1,g1])]; end if; end do;

$>$ for e from 1 to nops(P) do

$>$ Q[a]:=[op(Q[a]),limit(evalf( P[e][1]/P[e][2] ),x=0)]; end do;

$>$ end do; [seq(op(Q[a]),a=1..nops(F))];

$>$ end proc:\medskip

\section{Calculation}

Next, we give some examples illustrating the performance of the
algorithm:\medskip

\begin{enumerate}
\item limite(6*x\symbol{94}3*y,2*x\symbol{94}4+y\symbol{94}4,20);
[2.033104508, -2.033104508, 0.]. Consequently, $%
\lim_{(x,y)->(0,0)}6x^{3}y/(2x^{4}+y^{4})$ does not exist.

\item limite((x\symbol{94}3+y\symbol{94}3),(x\symbol{94}2+x*y+y\symbol{94}%
2),20); [0., 0., 0.]. In this case $\lim_{(x,y)\rightarrow
(0,0)}(x^{3}+y^{3})/(x^{2}+xy+y^{2})$ exists and equals $0.$

\item limite(6*x\symbol{94}3*y,2*x\symbol{94}4+y\symbol{94}4,20);
[2.033104508, -2.033104508, 0.]. The limit $%
\lim_{(x,y)->(0,0)}6x^{3}y/(2x^{4}+y^{4})$ does not exist.

In the following the limit exists only in 4. and 5.

\item limite((x\symbol{94}4-y\symbol{94}2+3*x\symbol{94}2*y-x\symbol{94}2),(x%
\symbol{94}2+y\symbol{94}2),20); [-1., -1., -1.].

\item limite( (x\symbol{94}2-y\symbol{94}2), (x\symbol{94}2+y\symbol{94}%
2),10); [1., -1.].

\item limite((x\symbol{94}6-y\symbol{94}4+3*x\symbol{94}2*y\symbol{94}3-x%
\symbol{94}4*y),(x\symbol{94}4+y\symbol{94}4+x\symbol{94}2+y\symbol{94}%
2),20); [0., 0., 0., 0.]

\item limite(x,x\symbol{94}2+y\symbol{94}2,30); [undefined]

\item limite(y\symbol{94}4,x\symbol{94}4+3*y\symbol{94}4,50); [0.3333333333,
0.]

\item limite(6*x\symbol{94}3*y,2*x\symbol{94}4+y\symbol{94}4,10); [0.,
2.033104508, -2.033104508].

\item limite(x\symbol{94}4*y\symbol{94}4, (x\symbol{94}8+y\symbol{94}8)%
\symbol{94}3, 20); [0., 0., Float(infinity), Float(infinity)].
\end{enumerate}

\section{Discussion and Conclusions}

The algorithm developed in this article provides a method for computing
limits of quotients of real polynomials in two variables which has proven to
be more powerful in handling these type of limits than other existing
algorithms. The following examples compare the performance of the routine\
here presented and that of \textit{Maple 12.\medskip }

\begin{enumerate}
\item Example number 5 in the previous section shows that $%
\lim_{(x,y)\rightarrow (0,0)}(x^{4}-y^{2}+3x^{2}y-x^{2})/(x^{2}+y^{2})$ is
equal to $-1.$ However, \textit{Maple 12} is uncapable to compute it.

\item Example 8 shows that $\lim_{(x,y)\rightarrow (0,0)}y^{4}/(x^{4}+3y^{4})
$ does not exist, an example that \textit{Maple 12} computes successfully.

\item $\lim_{(x,y)\rightarrow
(0,0)}(x^{6}-y^{4}+3x^{2}y-x^{4}y)/(x^{4}+y^{4}+x^{2}+y^{2})$ is equal to
zero, a problem that defeats Maples%
%TCIMACRO{\U{b4} }%
%BeginExpansion
\'{}
%EndExpansion
routine.
\end{enumerate}

The theoretical method for computing limits developed in this article
applies to quotients of two real analytic functions. However, the algorithm
was only implemented for polynomials. The next logical step would be to
extend this algorithm to cover this general case. 

In a sequel article we will develop a more general method dealing with
limits of quotiens of real analytic functions in several variables.

\section{Acknowledgements}

We are grateful to the Universidad Nacional of Colombia and Universidad
Eafit, for their invaluable support.

\end{document}